\documentclass[reqno]{amsart}
\allowdisplaybreaks[1]
\usepackage{mathtools} \mathtoolsset{showonlyrefs}
\usepackage{enumitem}

\usepackage{setspace}
\setdisplayskipstretch{1}
\usepackage[T1]{fontenc}
\usepackage{lmodern}

\usepackage{amsmath,amsfonts,amsbsy,amsgen,amscd,mathrsfs,amssymb,amsthm}
\usepackage{bm,bbm}

\usepackage[colorlinks=true,
            linkcolor=blue,
            citecolor=blue,
            urlcolor=blue]{hyperref}

\usepackage{etoolbox} %Clever Refencing inside enumerate environment.
\usepackage[capitalize,noabbrev]{cleveref}

\usepackage{tikz}
\theoremstyle{plain}
\newtheorem{theorem}{Theorem}[section]
\newtheorem{lemma}[theorem]{Lemma}

\newtheorem{conjecture}[theorem]{Conjecture}

\theoremstyle{definition}

\theoremstyle{remark}

\newcommand{\abs}[1]{\ensuremath{\left\lvert #1 \right\rvert}}

\newcommand{\floor}[1]{\ensuremath{\left\lfloor #1 \right\rfloor}}

\DeclareMathOperator{\Span}{span}

\title{Matroid flat counts are not unimodal}
\author{Alexander Divoux}
\address{Program in Applied and Computational Mathematics, Princeton University}
\email{adivoux@princeton.edu}
\author{Chayim Lowen}
\address{Department of Mathematics, Princeton University}
\email{chayiml@princeton.edu}
\author{Shouda Wang}
\address{Program in Applied and Computational Mathematics, Princeton University}
\email{shoudawang@princeton.edu}
\begin{document}
\begin{abstract}
    We give counterexamples to Rota's 1970 conjecture that the sequence counting flats of varying rank in a matroid is unimodal. More specifically, inspired by Larson's recent disproof of the stronger log-concavity conjecture of Mason, we explain a mechanism which turns failures of log-concavity for flats into failures of unimodality under suitable conditions.
\end{abstract}

\maketitle

\section{Introduction}
Let $M$ be a matroid of rank $r$. For $0 \le i \le r$, let $W_i$ denote the number of rank $i$ flats in $M$. A long-standing conjecture of Rota \cite{rota70} asserts that the sequence $W_0,\dots, W_r$ is unimodal.
\begin{conjecture}
\label[conjecture]{conj}
    The sequence $(W_i)_{0\leq i\leq r}$ is unimodal for every matroid. That is, there exists an integer $0 \le k \le r$ such that 
    $$
    W_0\leq \cdots \leq W_{k-1}\leq  W_k \geq W_{k+1}\geq \cdots \geq W_r.
    $$
\end{conjecture}
A strengthening due to Mason \cite{mason72} proposed that $(W_i)_i$ should in fact be log-concave. Rota and Mason's conjectures have attracted considerable attention, see, for example, Conjectures 8.2.4(i)–(ii) in \cite{aigner87}, Conjectures 15.2.3 and 15.2.4(i) in \cite{Oxley}, and Problem 25(c) in \cite{stanley00}. The past decade has seen the resolution of several closely related conjectures on the unimodality and log-concavity of sequences coming from matroids, due to Heron, Rota, Welsh, and Mason \cite{heron1972matroid,rota70,welsh1976matroid,mason72}. The breakthrough works \cite{huh12a, huh12b, Huh18,huh20,anari2024logconcave} that resolved these questions have given birth to the beautiful theories of combinatorial Hodge theory and Lorentzian polynomials. All of this powerful machinery, however, has shed little light on  Rota's or Mason's conjectures for flats.

Rota's and Mason's conjectures for flats are known to hold for some special classes of matroids. Mason's conjecture was proven for braid matroids in \cite{Harper} and \cite{lieb}, and Rota's conjecture for perfect matroid designs in \cite{Y70}. The special case of Mason's conjecture at $W_2$ asserts that $W_2^2\ge W_1W_3$, and is known as the points-lines-planes conjecture. A stronger inequality was proved by Stonesifer \cite{Stone75} for graphic matroids, and was later extended by Seymour \cite{seymour} to matroids in which every rank two flat has size at most 4. The same inequality was verified for all matroids of rank 4 on at most 11 elements in \cite{Matsumoto}.

Substantial progress on \Cref{conj} was made by Huh and Wang \cite{Huh1}, who showed for realizable matroids that unimodality holds for the first half of the sequence. More precisely, they proved that any realizable matroid of rank $r$ satisfies
\[
    W_0 \leq W_1 \leq \dots \leq W_{\floor{r/2}}.
\]
The same inequalities were proven for all matroids in \cite{huh2} by a difficult inductive argument. A significant simplification was later achieved in \cite{AHL25}.

Recently, Larson \cite{Larson26} disproved Mason's conjecture by showing that the graphic matroids of suitable generalized theta graphs violate log-concavity at the tail of the sequence. In the present paper, we disprove Rota's original unimodality conjecture.
\begin{theorem}\label{thm:main}
    Conjecture \ref{conj} is false.
\end{theorem}
The construction of our counterexamples starts with matroids which violate the log-concavity of flats to a large degree, then boosts this deficiency to a failure of unimodality. To accomplish this, we apply the $q$-lift construction introduced by Whittle \cite{Whittle89} to Larson's generalized theta graph matroids, which provide us with many strong violations of log-concavity.

\subsection*{Overview of later sections.}
In \cref{sec:prelim}, we review notation and give some intuition for how our construction works. In \cref{sec:lcc}, we review Larson's construction of matroids which violate Mason's conjecture. In \Cref{sec:q_lifts}, we describe the $q$-lift construction and compute its effect on flat counts. Finally, in \Cref{sec:proof_of_main}, we combine the results of these two sections to prove \Cref{thm:main}.

\section{Preliminaries}\label{sec:prelim}
We refer the reader to \cite{Oxley,White} for background on matroids.
For a matroid $M$, we write $E(M)$ for the ground set of $M$ and $r_M$ for the rank function of $M$.
Recall that a \emph{flat} of $M$ of rank $i$ is a maximal subset of $E(M)$ with rank $i$.
A \emph{cyclic set} of $M$ is any union of circuits of $M$.
The cyclic sets of $M$ are precisely the complements of the flats of the dual matroid $M^*$ and vice versa. 
 The rank of a subset $S \subseteq E(M)$ in the dual matroid $M^*$ can be expressed via the rank function of $M$, and is given by
\[
    r_{M^*}(S) = r_M(E \setminus S) + \abs{S} - r_M(E).
\]
We write $W_i(M)$ for the number of flats of $M$ of rank $i$, known as the $i$-th \emph{Whitney number of the second kind}. 
When there is no risk of ambiguity, we will sometimes write $E$ in place of $E(M)$ and  $W_i$ in place of $W_i(M)$. We will also write $r$ or $r(M)$ in place of 
$r_M(E)$.

We follow standard asymptotic notation as set out in \cite{Knuth}. For functions $f,g$ on $\mathbb N$ valued in the nonnegative reals, we write $f = O(g)$ if there exist a positive constant $C$ and an integer $n_0$ such that $f(n) \le Cg(n)$ for all $n \ge n_0$. We write $f = \Omega(g)$ to mean $g = O(f)$. The notation $f = \Theta(g)$ means that $f = O(g)$ and $g = O(f)$ hold simultaneously.

A sequence $a_0,\dots, a_m$ of real numbers is \emph{unimodal} if 
for some $0 \leq k \leq m$ we have 
\[
    a_0 \leq a_1 \leq \dots \leq a_{k-1} \leq a_k \geq a_{k+1} \geq \dots \geq a_{m-1} \geq a_m.
\]
Equivalently, $a_j \geq \min(a_{i}, a_{k})$ for all $ i<j < k$. A sequence $a_0, \dots, a_m$ of nonnegative  numbers is \textit{log-concave} if $a_i^2 \ge a_{i-1}a_{i+1}$ for every $0 < i < m$ and contains no internal zeros. It is well-known and easy to prove that log-concave sequences are unimodal.

A more precise observation is that a sequence $a_0, \dots, a_n$ of positive real numbers is log-concave if and only if the sequence $i \mapsto c^i a_i$ is unimodal for all constants $c > 0$. The aforementioned conjecture of Rota asserts that the sequence $W_0, \dots, W_r$ is unimodal. Given the observation above, a plausible approach to producing a counterexample is to start with a matroid for which log-concavity of $(W_i)_i$ fails and to construct from it a matroid with Whitney numbers $(c^i W_i)_i$, where $c$ is suitably chosen to break unimodality.

While the present authors are not aware of a construction on matroids that has precisely the above effect on flat counts, the $q$-lift construction introduced by Whittle in \cite{Whittle89} has almost this effect for $\mathbb F_q$-representable matroids. See \Cref{sec:q_lifts} for details. The numerical subtleties involved are resolved by ensuring that the failure of log-concavity in the starting matroid---as measured by the ratio $W_{k-1}W_{k+1}/W_k^2$---is sufficiently large. Since the $q$-lift construction takes as input an $\mathbb{F}_q$-realizable matroid, it is also important to us that Larson's counterexamples are regular matroids, which are realizable over all fields.

\section{Log-concavity}\label{sec:lcc}

For $t \ge 1$, let $G_t$ denote the graphic matroid of the generalized theta graph consisting of two endpoint vertices and four internally disjoint paths between them, three of length $t$ and one of length $1$. The matroid $G_t$ has $3t+1$ elements and rank $3t-2$. In \cite{Larson26}, the matroid $G_{26}$ was used to disprove Mason's conjecture, for which it was shown that
\[
    \frac{W_{75}\cdot W_{73}}{W_{74}^2} = \frac{18551\cdot 52954525}{983775^2} \approx 1.015.
\]
The computation there easily generalizes to show that this ratio can be made arbitrarily large, which will be used in the proof of Theorem \ref{thm:main}. We include the proof here for completeness and because we need a slightly more detailed statement. 

\begin{lemma}\label[lemma]{lemma:gap}
    For any $\rho > 1$, there exists a simple regular matroid $M$ and $3 < k < r(M)$ such that $W_{k-2} > W_{k-1} > W_k > W_{k+1}$ but $W_{k-1}W_{k+1}/W_k^2 > \rho$.
\end{lemma}
\begin{proof}
    Let $t\ge 3$ be a positive integer to be chosen later. Let $M = G_t$, and let $k = r(M)-2 = 3t-4 > 3$. In the dual $M^*$, the four series classes of $M$ become parallel classes $P_0, P_1, P_2, P_3$, with $\abs{P_0}=1$ and $\abs{P_1}=\abs{P_2}=\abs{P_3}=t$. Simplifying these parallel classes yields the uniform matroid $U_{3,4}$, so each circuit of $M^*$ is either a two-element subset of one of $P_1, P_2, P_3$, or has size four and consists of one element from each parallel class. For every flat $F$ of $M$ of rank $0 \le j \le r(M)$, the corresponding cyclic set $S = E\setminus F$ of $M^*$ satisfies
    \[
        \abs{S}-r_{M^*}(S) = r(M) - r_M(F) = r(M) - j
    \]
    by the dual rank formula. We will refer to the quantity $\abs{S} - r_{M^*}(S)$ as the \emph{nullity} of $S$ (in $M^*$). By the equation above, computing $W_{k-2}, W_{k-1}, W_k$, and $W_{k+1}$ amounts to counting cyclic sets $S$ of $M^*$ of nullity $4,3,2,$ and $1$, respectively.

    If $S$ has nullity $1$, then $S$ is a circuit. By the characterization above,
    \[
        W_{k+1}=3\binom t2 + t^3 = \Theta(t^3).
    \]

    Next suppose that $S$ has nullity $2$. Since $r(M^*)=3$, we have $3\le |S|\le 5$. If $|S|=3$, then $S$ consists of three elements from one of $P_1,P_2,P_3$, for which there are $3\binom t3$ choices. If $|S|=4$, then $S$ consists of two elements from each of two among $P_1,P_2,P_3$, for which there are $3{\binom t2}^2$ choices. If $|S|=5$, then $S$ consists of one element from each of $P_0, P_1, P_2, P_3$ and one additional element from one of $P_1,P_2,P_3$. Here there are $3t^2\binom t2$ choices. Therefore,
    \[
        W_k = 3\binom t3 + 3{\binom t2}^2 + 3t^2\binom t2  = \Theta(t^4).
    \]

    Now suppose that $S$ has nullity $3$. Since $r(M^*)=3$, we have $4\le|S|\le6$. If $S$ does not contain $P_0$, then there are three possibilities for the content of $S$:
    \begin{itemize}
        \item four elements from one of $P_1, P_2, P_3$; or
        \item three elements from one of $P_1, P_2, P_3$ and two elements from another; or
        \item two elements from each of $P_1, P_2, P_3$.
    \end{itemize}
    If $S$ does contain $P_0$,  then it has rank $3$ because all circuits containing $P_0$ have rank $3$, and so $|S|=6$. In this case, numbers of elements chosen from $P_1, P_2, P_3$ are $3, 1, 1$ or $2, 2, 1$, up to permutation.  It follows that
    \[
        W_{k-1} = 3\binom t4 + 6\binom t3\binom t2 + {\binom t2}^3 + 3t^2 \binom t3 + 3t {\binom t2}^2 =\Theta(t^6).
    \]

    Finally, we give a lower bound on $W_{k-2}$.
    To get a cyclic set of nullity 4, we can choose three elements from one of $P_1, P_2, P_3$ and two elements from each of the other two. We therefore get $W_{k-2}=\Omega(t^7)$ via the inequality
    \[
        W_{k-2}\ge 3\binom t3{\binom t2}^2 = \Theta(t^7).
    \]
    
    Combining the four asymptotic estimates above, we see immediately that the inequality $W_{k-2}>W_{k-1}>W_k>W_{k+1}$ holds for all sufficiently large $t$. Moreover, the ratio $W_{k-1}W_{k+1}/W_k^2 = \Theta(t)$  tends to infinity in $t$, so is eventually greater than $\rho$. Since $M$ is graphic, it is also regular, which completes the proof. 
\end{proof}

\section{$q$-lifts}\label{sec:q_lifts}
Throughout this section, we let $q$ be a prime power and let $\mathbb F_q$ be the finite field of cardinality $q$. Let $M$ be a simple $\mathbb F_q$-representable matroid of rank $r$ on $n$ elements.
Fixing an $\mathbb F_q$-representation of $M$, we identify $E(M)$ with the set of vectors in this vector representation.
Embed $\mathbb F_q^r$ into $\mathbb F_q^{r+1} = \mathbb F_q\times \mathbb F_q^r$. We consider the matroid $M'$ consisting of the vector $a = (1, 0, \dots, 0) \in \mathbb F_q^{r+1}$ and the vectors  $(t,v)$, where $t$ runs over $\mathbb F_q$ and $v$ runs over $E(M)$.
This operation was introduced by Whittle \cite{Whittle89}, where $M'$ is called the \textit{$q$-lift} of $M$ and $a$ is called the \textit{apex}. 

The isomorphism class of $M'$ is independent of the choice of apex $a$ in the description above. By contrast, Oxley and Whittle \cite{OxleyWhittle00} established that different initial representations of $M$ can yield non-isomorphic $q$-lifts of $M$. Nevertheless, it was shown by Bonin and Qin \cite{BoninQin01} that any two $q$-lifts of $M$ have the same Tutte polynomial. We will see in what follows that they also share the same Whitney numbers of the second kind.

In \cref{lemma:q-lift}, we count the number of flats of any $q$-lift $M'$ of $M$ in terms of the number of flats of $M$. This appears as \cite[Lemma 4]{BoninQin01}, but we include a proof here for completeness. Recall that every flat in a linearly-realized matroid can be written as the intersection of the set of vectors representing its ground set with some linear subspace of the ambient vector space.

\begin{lemma}\label[lemma]{lemma:q-lift}
    If $M$ is a simple $\mathbb F_q$-representable matroid of rank $r$ and $M'$ is a $q$-lift of $M$, then for every $1\leq i\leq r$ we have
    \[
        W_i(M') = q^i W_i (M) + W_{i-1}(M).
    \]
\end{lemma}
\begin{proof}
Throughout the proof, we use $\Span(A)$ to denote the  linear subspace spanned by a set $A$ of vectors.
Denote by $ \pi: \mathbb F_q^{r+1}\rightarrow\mathbb F_q^r$  the projection  given by $\pi(t,v) = v$.

    Fix $i \ge 1$. We claim that any rank $i$ flat $F$ of $M'$ can be uniquely written in one of the following two forms, according to whether it contains the apex $a$ or not.

    \textit{Case 1}: $a\in F$, in which case there exists a unique flat $G$ in $M$ of rank $i-1$ such that
    \[
        F = \{a\}\cup \{(t,v): t\in\mathbb F_q,\, v\in G\}.
    \]
    This equality forces $G = E(M) \cap \pi(\Span(F))$, which gives uniqueness of $G$. Conversely, if we use this as the definition of $G$, we  see that $G$ is a flat. For this choice of $G$, the inclusion
    \[
        F \subseteq \{a\}\cup \{(t,v): t\in\mathbb F_q,\, v\in G\}.
    \]
    is clear. For the reverse inclusion, let  $(t, v)$ be in the right-hand side. Since $(t, v) \in \pi^{-1}(\pi(\Span(F))) = \Span F$, we have $(t, v) \in E(M') \cap \Span F = F$. Finally, from the formula for $F$, we have $\Span(F) = \Span(a) \oplus \Span(G)$, and therefore $r_{M}(G)=r_{M'}(F) -1=i-1$.

    \textit{Case 2}: $a\notin F$, in which case there exists a unique flat $G$ in $M$ of rank $i$ and a unique linear functional $\alpha: \Span (G)\rightarrow \mathbb F_q$ such that 
    \[
        F =  \{(\alpha(v),v): v\in G\}.
    \]
    Note that since $a\notin F$, we have $\pi|_{\Span(F)}$ is injective. We can therefore recover both $G$ and $\alpha$ from $F$, giving uniqueness.

    To establish existence, define $G  = E(M) \cap \pi(\Span(F))$
    and let $\alpha: \Span(G)\rightarrow \mathbb F_q $ be given by 
    \[
         \alpha (v) = \pi_0(\pi|_{\Span(F)}^{-1} (v))
    \]
    for all $v\in  \Span(G) $, where $\pi_0:\mathbb F_q^{r+1}\rightarrow\mathbb F_q$ is the projection to the first coordinate $\pi_0(t,v) = t$.
    By construction, $G$ is a flat of $M$. The non-degeneracy of $\pi|_{\Span (F)}$  shows that $r_M(G) = r_{M'}(F)=i$. Furthermore, we have
     \[
         \pi|_{\Span (F)}^{-1}(v) = (\alpha(v),v)
     \]
     for all $v\in \Span (G)$, showing that $G$ and $\alpha$ satisfy $F = \{(\alpha(v),v): v\in G\}$.
     
     Having understood all the flats of $M'$, we can now count them. By the above case analysis, a flat of $M'$ of rank $i$ is given (uniquely) by either a flat $G$ of $M$ of rank $i-1$ or else a pair $(G, \alpha)$ where $G$ is a rank $i$ flat of $M$ and $\alpha$ is a linear functional on $\Span(G)$. Let us note that, conversely, any given rank $i-1$ flat $G$ in $M$ defines a rank $i$ flat $F$ in $M'$ as in Case 1. Similarly, any pair $(G,\alpha)$ where $G$ is a rank $i$ flat of $M$ and $\alpha$ is a linear functional on $\Span(G)$ defines a rank $i$ flat $F$ in $M'$ as in Case 2. Since in the latter case there are $q^{r_{M}(G)} = q^i$ functionals on $\Span(G)$, we deduce the equation
     \[
        W_i(M') = W_{i-1}(M) + q^i W_i(M),
     \]
     as desired.
\end{proof}

\section{Proof of the main theorem}\label{sec:proof_of_main}

\begin{proof}[Proof of Theorem~\ref{thm:main}]
    Let $M$ be any matroid satisfying the conditions in \Cref{lemma:gap} for the value $\rho = 4$. In particular, $M$ satisfies $W_{k-2} > W_{k-1} > W_k > W_{k+1}$ and $W_{k-1}W_{k+1}/W_k^2 > 4$ for some $3 < k < r(M)$. It follows that
    \[
        \frac{W_{k-1}}{W_k} > 4\cdot \frac{W_k}{W_{k+1}} > 4.
    \]
    By Bertrand's postulate (see, e.g., \cite[Theorem 418]{HardyWright08}),  there exists a prime number $p$  such that
    \[
         p\in \left(\frac12\cdot\frac{W_{k-1}}{W_k},\frac{W_{k-1}}{W_k}\right).
    \]
    Since $M$ is regular, it is $\mathbb F_p$-representable, and hence we may choose a $p$-lift $M'$ of $M$. We write $W_i'$ for the number of flats of rank $i$ in $M'$. We now show that the differences $W_{k-1}' - W_k'$ and $W_{k+1}' - W_k'$ are both positive. By \cref{lemma:q-lift}, we have
    \[
        W_{k-1}'-W_k' = p^{k-1}(W_{k-1}-pW_k) + W_{k-2} - W_{k-1}.
    \]
    By our choice of $p$, we have $W_{k-1} > pW_k$. Since $W_{k-2} > W_{k-1}$, it follows that $W_{k-1}'-W_k' > 0$. For the second difference, note that our choice of $p$ implies $p > \frac{1}{2} \cdot \frac{W_{k-1}}{W_k} > 2 \cdot \frac{W_k}{W_{k+1}}$, so by \cref{lemma:q-lift} we have
    \begin{align*}
        W_{k+1}' - W_k' &= p^{k}(p W_{k+1} - W_k) + W_{k} - W_{k-1}\\
        &> p^k(2W_k - W_k) + W_k - 2p W_k \\
        &= (p^k - 2p + 1)W_k > 0.
    \end{align*}
    Therefore, $W_{k-1}' > W_k'$ and $W_k' <  W_{k+1}'$, so the sequence of Whitney numbers of the second kind in $M'$ is not unimodal.
\end{proof}

\newpage

\subsection*{How the counterexamples were found.} The starting point of this paper was an attempt to study the convexity of $(W_i^{-1})_i$, which is weaker than log-concavity but stronger than unimodality. We used ChatGPT 5.6 Pro to search for a counterexample to this weaker property. Based on a suggestion of the authors, ChatGPT 5.6 Pro returned a non-convex example built using a $q$-lift. We quickly realized this construction could in fact be adapted and generalized to give counterexamples to unimodality. The content of the present paper was written entirely by the authors.

\subsection*{Acknowledgments.} The authors are grateful to Matt Larson and Paul Seymour for helpful discussions. Part of this work was carried out at the 2026 \emph{Combinatorics at the Confluence} conference. The authors are grateful to the organizers and to Carnegie Mellon University for their hospitality.

\bibliographystyle{alpha}
\bibliography{references}

@article {Matsumoto,
    AUTHOR = {Matsumoto, Yoshitake and Moriyama, Sonoko and Imai, Hiroshi
              and Bremner, David},
     TITLE = {Matroid enumeration for incidence geometry},
   JOURNAL = {Discrete Comput. Geom.},
  FJOURNAL = {Discrete \& Computational Geometry. An International Journal
              of Mathematics and Computer Science},
    VOLUME = {47},
      YEAR = {2012},
    NUMBER = {1},
     PAGES = {17--43},
      ISSN = {0179-5376,1432-0444},
   MRCLASS = {05B35 (51A99 52C35 52C40)},
  MRNUMBER = {2886089},
MRREVIEWER = {Winfried\ Hochst\"attler},
       DOI = {10.1007/s00454-011-9388-y},
       URL = {https://doi.org/10.1007/s00454-011-9388-y},
}

@misc{Larson26,
      title={Counterexamples to two conjectures about matroids}, 
      author={Matt Larson},
      year={2026},
      eprint={2607.02208},
      archivePrefix={arXiv},
      primaryClass={math.CO},
      url={https://arxiv.org/abs/2607.02208}, 
}

@incollection {rota70,
    AUTHOR = {Rota, Gian-Carlo},
     TITLE = {Combinatorial theory, old and new},
 BOOKTITLE = {Actes du {C}ongr\`es {I}nternational des {M}ath\'ematiciens
              ({N}ice, 1970), {T}ome 3},
     PAGES = {229--233},
 PUBLISHER = {Gauthier-Villars \'Editeur, Paris},
      YEAR = {1971},
   MRCLASS = {05B25},
  MRNUMBER = {505646},
}

@incollection {heron1972matroid,
    AUTHOR = {Heron, A. P.},
     TITLE = {Matroid polynomials},
 BOOKTITLE = {Combinatorics ({P}roc. {C}onf. {C}ombinatorial {M}ath.,
              {M}ath. {I}nst., {O}xford, 1972)},
     PAGES = {164--202},
 PUBLISHER = {Inst. Math. Appl., Southend-on-Sea},
      YEAR = {1972},
   MRCLASS = {05B35},
  MRNUMBER = {340058},
MRREVIEWER = {Ann\ Miller},
}

@book {welsh1976matroid,
    AUTHOR = {Welsh, D. J. A.},
     TITLE = {Matroid theory},
    SERIES = {L. M. S. Monographs},
    VOLUME = {No. 8},
 PUBLISHER = {Academic Press [Harcourt Brace Jovanovich, Publishers],
              London-New York},
      YEAR = {1976},
     PAGES = {xi+433},
   MRCLASS = {05B35},
  MRNUMBER = {427112},
MRREVIEWER = {W.\ T.\ Tutte},
}

@article {anari2024logconcave,
    AUTHOR = {Anari, Nima and Liu, Kuikui and Oveis Gharan, Shayan and
              Vinzant, Cynthia},
     TITLE = {Log-concave polynomials {III}: {M}ason's ultra-log-concavity
              conjecture for independent sets of matroids},
   JOURNAL = {Proc. Amer. Math. Soc.},
  FJOURNAL = {Proceedings of the American Mathematical Society},
    VOLUME = {152},
      YEAR = {2024},
    NUMBER = {5},
     PAGES = {1969--1981},
      ISSN = {0002-9939,1088-6826},
   MRCLASS = {05B35 (05A20 52A41)},
  MRNUMBER = {4728467},
       DOI = {10.1090/proc/16724},
       URL = {https://doi.org/10.1090/proc/16724},
}

@incollection {mason72,
    AUTHOR = {Mason, J. H.},
     TITLE = {Matroids: unimodal conjectures and {M}otzkin's theorem},
 BOOKTITLE = {Combinatorics ({P}roc. {C}onf. {C}ombinatorial {M}ath.,
              {M}ath. {I}nst., {O}xford, 1972)},
     PAGES = {207--220},
 PUBLISHER = {Inst. Math. Appl., Southend-on-Sea},
      YEAR = {1972},
   MRCLASS = {05B35},
  MRNUMBER = {349445},
MRREVIEWER = {W.\ Dorfler},
}

@book {Oxley,
    AUTHOR = {Oxley, James},
     TITLE = {Matroid theory},
    SERIES = {Oxford Graduate Texts in Mathematics},
    VOLUME = {21},
   EDITION = {Second},
 PUBLISHER = {Oxford University Press, Oxford},
      YEAR = {2011},
     PAGES = {xiv+684},
      ISBN = {978-0-19-960339-8},
   MRCLASS = {05-01 (05B35 90C27)},
  MRNUMBER = {2849819},
MRREVIEWER = {Maruti\ M.\ Shikare},
       DOI = {10.1093/acprof:oso/9780198566946.001.0001},
       URL = {https://doi.org/10.1093/acprof:oso/9780198566946.001.0001},
}

@book{White,
    editor={Neil White},
    place={Cambridge},
    series={Encyclopedia of Mathematics and its Applications},
    title={Theory of Matroids}, 
    publisher={Cambridge University Press},
    year={1986},
    collection={Encyclopedia of Mathematics and its Applications}
}

@article {Whittle89,
    AUTHOR = {Whittle, Geoff},
     TITLE = {{$q$}-lifts of tangential {$k$}-blocks},
   JOURNAL = {J. London Math. Soc. (2)},
  FJOURNAL = {Journal of the London Mathematical Society. Second Series},
    VOLUME = {39},
      YEAR = {1989},
    NUMBER = {1},
     PAGES = {9--15},
      ISSN = {0024-6107,1469-7750},
   MRCLASS = {51E20 (05B25 05B35)},
  MRNUMBER = {989915},
MRREVIEWER = {Ulrich\ Faigle},
       DOI = {10.1112/jlms/s2-39.1.9},
       URL = {https://doi.org/10.1112/jlms/s2-39.1.9},
}

@article {OxleyWhittle00,
    AUTHOR = {Oxley, James and Whittle, Geoff},
     TITLE = {On the non-uniqueness of {$q$}-cones of matroids},
   JOURNAL = {Discrete Math.},
  FJOURNAL = {Discrete Mathematics},
    VOLUME = {218},
      YEAR = {2000},
    NUMBER = {1-3},
     PAGES = {271--275},
      ISSN = {0012-365X,1872-681X},
   MRCLASS = {05B35},
  MRNUMBER = {1754341},
       DOI = {10.1016/S0012-365X(99)00358-1},
       URL = {https://doi.org/10.1016/S0012-365X(99)00358-1},
}

@book {HardyWright08,
    AUTHOR = {Hardy, G. H. and Wright, E. M.},
     TITLE = {An introduction to the theory of numbers},
   EDITION = {Fourth},
 PUBLISHER = {Oxford University Press, Oxford},
      YEAR = {1960},
      ISBN = {0-19-853310-7},
}

@article {BoninQin01,
    AUTHOR = {Bonin, Joseph E. and Qin, Hongxun},
     TITLE = {Tutte polynomials of {$q$}-cones},
   JOURNAL = {Discrete Math.},
  FJOURNAL = {Discrete Mathematics},
    VOLUME = {232},
      YEAR = {2001},
    NUMBER = {1-3},
     PAGES = {95--103},
      ISSN = {0012-365X,1872-681X},
   MRCLASS = {05B35},
  MRNUMBER = {1823624},
       DOI = {10.1016/S0012-365X(00)00343-5},
       URL = {https://doi.org/10.1016/S0012-365X(00)00343-5},
}

@incollection {aigner87,
    AUTHOR = {Aigner, Martin},
     TITLE = {Whitney numbers},
 BOOKTITLE = {Combinatorial geometries},
    SERIES = {Encyclopedia Math. Appl.},
    VOLUME = {29},
     PAGES = {139--160},
 PUBLISHER = {Cambridge Univ. Press, Cambridge},
      YEAR = {1987},
      ISBN = {0-521-33339-3},
   MRCLASS = {05B35},
  MRNUMBER = {921072},
}

@incollection {stanley00,
    AUTHOR = {Stanley, Richard P.},
     TITLE = {Positivity problems and conjectures in algebraic
              combinatorics},
 BOOKTITLE = {Mathematics: frontiers and perspectives},
     PAGES = {295--319},
 PUBLISHER = {Amer. Math. Soc., Providence, RI},
      YEAR = {2000},
      ISBN = {0-8218-2070-2},
   MRCLASS = {05-02 (05E05 26C10 52B05)},
  MRNUMBER = {1754784},
MRREVIEWER = {Timothy\ Y.\ Chow},
}

@book {Y70,
    AUTHOR = {Young, Hobart Peyton},
     TITLE = {Equicardinal Matroids and Matroid Designs},
      NOTE = {Thesis (Ph.D.)--University of Michigan},
 PUBLISHER = {ProQuest LLC, Ann Arbor, MI},
      YEAR = {1970},
     PAGES = {212},
   MRCLASS = {99-05},
  MRNUMBER = {2620412},
       URL =
              {http://gateway.proquest.com/openurl?url_ver=Z39.88-2004&rft_val_fmt=info:ofi/fmt:kev:mtx:dissertation&res_dat=xri:pqdiss&rft_dat=xri:pqdiss:7115351},
}

@article {seymour,
    AUTHOR = {Seymour, P. D.},
     TITLE = {On the points-lines-planes conjecture},
   JOURNAL = {J. Combin. Theory Ser. B},
  FJOURNAL = {Journal of Combinatorial Theory. Series B},
    VOLUME = {33},
      YEAR = {1982},
    NUMBER = {1},
     PAGES = {17--26},
      ISSN = {0095-8956,1096-0902},
   MRCLASS = {05B35},
  MRNUMBER = {678168},
MRREVIEWER = {John\ H.\ Mason},
       DOI = {10.1016/0095-8956(82)90054-5},
       URL = {https://doi.org/10.1016/0095-8956(82)90054-5},
}

@article {Stone75,
    AUTHOR = {Stonesifer, J. Randolph},
     TITLE = {Logarithmic concavity for edge lattices of graphs},
   JOURNAL = {J. Combinatorial Theory Ser. A},
  FJOURNAL = {Journal of Combinatorial Theory. Series A},
    VOLUME = {18},
      YEAR = {1975},
     PAGES = {36--46},
      ISSN = {0097-3165},
   MRCLASS = {05B35},
  MRNUMBER = {376407},
MRREVIEWER = {J.\ A.\ Bondy},
       DOI = {10.1016/0097-3165(75)90064-3},
       URL = {https://doi.org/10.1016/0097-3165(75)90064-3},
}

@article {Harper,
    AUTHOR = {Harper, L. H.},
     TITLE = {Stirling behavior is asymptotically normal},
   JOURNAL = {Ann. Math. Statist.},
  FJOURNAL = {Annals of Mathematical Statistics},
    VOLUME = {38},
      YEAR = {1967},
     PAGES = {410--414},
      ISSN = {0003-4851},
   MRCLASS = {60.10},
  MRNUMBER = {211432},
MRREVIEWER = {R.\ Fischler},
       DOI = {10.1214/aoms/1177698956},
       URL = {https://doi.org/10.1214/aoms/1177698956},
}

@article {lieb,
    AUTHOR = {Lieb, Elliott H.},
     TITLE = {Concavity properties and a generating function for {S}tirling
              numbers},
   JOURNAL = {J. Combinatorial Theory},
  FJOURNAL = {Journal of Combinatorial Theory},
    VOLUME = {5},
      YEAR = {1968},
     PAGES = {203--206},
      ISSN = {0021-9800},
   MRCLASS = {05.10},
  MRNUMBER = {230635},
MRREVIEWER = {Bernard\ Harris},
}

@article {Huh1,
    AUTHOR = {Huh, June and Wang, Botong},
     TITLE = {Enumeration of points, lines, planes, etc.},
   JOURNAL = {Acta Math.},
  FJOURNAL = {Acta Mathematica},
    VOLUME = {218},
      YEAR = {2017},
    NUMBER = {2},
     PAGES = {297--317},
      ISSN = {0001-5962,1871-2509},
   MRCLASS = {05A15 (52C10)},
  MRNUMBER = {3733101},
MRREVIEWER = {Michael\ J.\ Falk},
       DOI = {10.4310/ACTA.2017.v218.n2.a2},
       URL = {https://doi.org/10.4310/ACTA.2017.v218.n2.a2},
}

@article{huh2,
  author  = {Braden, Tom and Huh, June and Matherne, Jacob P. and Proudfoot, Nicholas and Wang, Botong},
  title   = {Singular {Hodge} theory for combinatorial geometries},
  journal = {Journal of the American Mathematical Society},
  note    = {To appear},
}

@misc{AHL25,
      title={A decomposition theorem for Lefschetz modules}, 
      author={Omid Amini and June Huh and Matt Larson},
      year={2025},
      eprint={2511.02026},
      archivePrefix={arXiv},
      primaryClass={math.AG},
      url={https://arxiv.org/abs/2511.02026}, 
}

@article {huh12a,
    AUTHOR = {Huh, June},
     TITLE = {Milnor numbers of projective hypersurfaces and the chromatic
              polynomial of graphs},
   JOURNAL = {J. Amer. Math. Soc.},
  FJOURNAL = {Journal of the American Mathematical Society},
    VOLUME = {25},
      YEAR = {2012},
    NUMBER = {3},
     PAGES = {907--927},
      ISSN = {0894-0347,1088-6834},
   MRCLASS = {14B05 (05B35 14C17)},
  MRNUMBER = {2904577},
MRREVIEWER = {Paolo\ Aluffi},
       DOI = {10.1090/S0894-0347-2012-00731-0},
       URL = {https://doi.org/10.1090/S0894-0347-2012-00731-0},
}

@article {huh12b,
    AUTHOR = {Huh, June and Katz, Eric},
     TITLE = {Log-concavity of characteristic polynomials and the {B}ergman
              fan of matroids},
   JOURNAL = {Math. Ann.},
  FJOURNAL = {Mathematische Annalen},
    VOLUME = {354},
      YEAR = {2012},
    NUMBER = {3},
     PAGES = {1103--1116},
      ISSN = {0025-5831,1432-1807},
   MRCLASS = {05B35},
  MRNUMBER = {2983081},
MRREVIEWER = {Talmage\ J.\ Reid},
       DOI = {10.1007/s00208-011-0777-6},
       URL = {https://doi.org/10.1007/s00208-011-0777-6},
}

@article {Huh18,
    AUTHOR = {Adiprasito, Karim and Huh, June and Katz, Eric},
     TITLE = {Hodge theory for combinatorial geometries},
   JOURNAL = {Ann. of Math. (2)},
  FJOURNAL = {Annals of Mathematics. Second Series},
    VOLUME = {188},
      YEAR = {2018},
    NUMBER = {2},
     PAGES = {381--452},
      ISSN = {0003-486X,1939-8980},
   MRCLASS = {05B35 (05E99 14C25 14T05)},
  MRNUMBER = {3862944},
MRREVIEWER = {Zvi\ Rosen},
       DOI = {10.4007/annals.2018.188.2.1},
       URL = {https://doi.org/10.4007/annals.2018.188.2.1},
}

@article {huh20,
    AUTHOR = {Br\"and\'en, Petter and Huh, June},
     TITLE = {Lorentzian polynomials},
   JOURNAL = {Ann. of Math. (2)},
  FJOURNAL = {Annals of Mathematics. Second Series},
    VOLUME = {192},
      YEAR = {2020},
    NUMBER = {3},
     PAGES = {821--891},
      ISSN = {0003-486X,1939-8980},
   MRCLASS = {52B40 (05A20 14T15)},
  MRNUMBER = {4172622},
MRREVIEWER = {Trygve\ Johnsen},
       DOI = {10.4007/annals.2020.192.3.4},
       URL = {https://doi.org/10.4007/annals.2020.192.3.4},
}

@incollection {Knuth,
    AUTHOR = {Knuth, Donald E.},
     TITLE = {Big omicron and big omega and big theta (1976)},
 BOOKTITLE = {Ideas that created the future---classic papers of computer
              science},
     PAGES = {441--446},
 PUBLISHER = {MIT Press, Cambridge, MA},
      YEAR = {2021},
      ISBN = {9780262045308},
   MRCLASS = {68Q25},
  MRNUMBER = {4679197},
}

\end{document}